\documentclass[reqno, amsthm, 11pt]{amsart}
\usepackage{tikz}
\usepackage{tikz-cd}
\usetikzlibrary{shapes,arrows,cd}
\usepackage{pdfsync}
\usepackage{mathrsfs,amssymb}

\usepackage{todonotes}

\usepackage{hyperref}

\usepackage{rotating}

\usepackage{amsfonts} 
\usepackage{amsmath}
\usepackage{amssymb}
\usepackage{latexsym}
\usepackage{mathrsfs}
\usepackage{graphicx}

\usepackage{epic,eepic,ecltree} % for drawing trees

\DeclareMathOperator\pref{pref} 
 
\DeclareMathOperator\red{red}

\newcommand{\GO}{
\Gpres{a,t}{atat^{-1}a^{-1}ta^{-1}t^{-1}=1}}

\newcommand{\Mpres}[2]{\mathrm{Mon}\langle #1\:|\:#2 \rangle}
\newcommand{\Mgen}[1]{\mathrm{Mon}\langle #1 \rangle}

\newcommand{\Gpres}[2]{\mathrm{Gp}\langle #1\:|\:#2 \rangle}
\newcommand{\Ipres}[2]{\mathrm{Inv}\langle #1\:|\:#2 \rangle}

\newcommand{\MO}{\mathrm{Mon}\langle A \:|\: u=v  \rangle}

\newcommand{\Fz}{Freiheitssatz}

\input xy
\xyoption{all}

\newcommand{\nat}{\natural}

\usepackage{xy} % for drawing pictures

\newcommand{\dedge}[1]{\ar@{--}[#1]}

\newcommand{\edge}[1]{\ar@{-}[#1]}

\newcommand{\lulab}[1]{\ar@{}[l]_<<{#1}}

\newcommand{\rulab}[1]{\ar@{}[r]^<<{#1}}

\newcommand{\ldlab}[1]{\ar@{}[l]^<<{#1}}

\newcommand{\rdlab}[1]{\ar@{}[r]_<<{#1}}

  \newcommand{\gr}{\mathscr{R}}

\newtheorem{thm}{Theorem}[section]

\newtheorem{cor}[thm]{Corollary}
\newtheorem{prop}[thm]{Proposition}
\newtheorem{lem}[thm]{Lemma}

\theoremstyle{definition}

\newtheorem{remark}[thm]{Remark}

\newtheorem{question}[thm]{Question}

\usepackage[active]{srcltx}

\usepackage{verbatim}

\usepackage{enumerate}

\begin{document}

\title[Word problem for one-relator inverse monoids]{
Undecidability of the word problem for one-relator inverse monoids via right-angled Artin subgroups of one-relator groups
}

\author{ROBERT D. GRAY}

\address{School of Mathematics, University of East Anglia, Norwich NR4 7TJ, England, UK}


\keywords{word problem, 
one-relator inverse monoid,  
one-relator group, 
membership problem in groups,  
right-angled Artin groups. \newline
\indent This research was supported by the EPSRC grant EP/N033353/1 ``Special inverse monoids: subgroups, structure, geometry, rewriting systems and the word problem''.}

\begin{abstract} 
We prove the following results: 
(1) There is a one-relator inverse monoid $\Ipres{A}{w=1}$ with undecidable word problem; and 
(2) There are one-relator groups with undecidable submonoid membership problem. The second of these results is proved by showing that for any finite forest the associated right-angled Artin group embeds into a one-relator group. 
Combining this with a result of Lohrey and Steinberg from 2008, we use this to prove that there is a one-relator group containing a fixed finitely generated submonoid in which the membership problem is undecidable.  To prove (1) a new construction is introduced which uses the one-relator group and submonoid in which membership is undecidable from (2) to construct a one-relator inverse monoid $\Ipres{A}{w=1}$ with undecidable word problem. Furthermore, this method allows the construction of an $E$-unitary one-relator inverse monoid of this form with undecidable word problem. The results in this paper answer a problem originally posed by 
Margolis, Meakin and Stephen in 1987. 
\end{abstract}

\maketitle

\section{Introduction} 
The study of algorithmic questions in algebra has a long history which has its origins in fundamental problems in logic and topology investigated in the beginning of the 20th century by Thue, Tietze, and Dehn. One of the most  fundamental algorithmic questions concerning algebraic structures is the word problem, which asks whether two terms written over generators represent the same element of the structure. Markov and Post proved independently that the word problem for finitely presented monoids is undecidable in general. This result was later extended to groups by Novikov and Boone (see e.g. \cite{BookAndOtto, Lyndon:2001lh} and the references therein). It is therefore natural to ask for which classes of monoids and groups the word problem is decidable. Many interesting families of finitely presented groups and monoids are known to have decidable word problem, such as hyperbolic groups (in the sense of Gromov), automatic groups and monoids, and monoids and groups which admit presentations by finite complete rewriting systems; see \cite{BookAndOtto, Bridson:1999fk, Epstein92, Lyndon:2001lh}. It is natural to imagine that the word problem might be decidable for groups or monoids which are, in some sense, close to being free. The class of one-relator groups falls into this category, and it is a consequence of classical work of Magnus \cite{Magnus32} that the word problem is decidable for one-relator groups. 

In contrast, it is still not known whether the word problem is decidable for one-relator monoids, that is, monoids defined by presentations of the form $\Mpres{A}{u=v}$ where $A$ is a finite alphabet and $u$ and $v$ are words in the free monoid $A^*$. This is one of the most important and fundamental longstanding open questions in this area.  While this problem remains open in general, it has been solved in a number of special cases in work of Adjan \cite{Adjan:1966bh}, Adjan and Oganessian \cite{Adyan:1987ys}, and Lallement \cite{Lallement:1974pi}.  In particular, in \cite{Adjan:1966bh} Adjan proved that the word problem for one-relator monoids of the form $\Mpres{A}{w=1}$ is decidable.

More recently, Ivanov, Margolis and Meakin \cite{Ivanov:2001kl} discovered an entirely new approach to the word problem for one-relator monoids, which uses ideas from the theory of inverse monoids. Inverse monoids are a class that lies between groups and general monoids. While groups are an algebraic abstraction of permutations,  and monoids of arbitrary mappings, inverse monoids correspond to partial bijections and provide an algebraic framework for studying partial symmetries of structures.  Utilising \cite{Adyan:1987ys}, Ivanov, Margolis and Meakin made the fundamental observation that a positive solution to the word problem for one-relator inverse monoid presentations of the form $\Ipres{A}{w=1}$ would imply a positive solution to the word problem for arbitrary one-relator monoids $\MO$; see \cite[Thoerem~2.2]{Ivanov:2001kl}. This result motivated subsequent work investigating the question of whether all one-relator inverse monoids of the form    $\Ipres{A}{w=1}$    have decidable word problem. This problem has now been shown to have a positive answer in several cases including when $w$ is: an idempotent word \cite{Margolis:1993tw}, a sparse word \cite{Hermiller:2010bs}, a one-relator surface group relation, a Baumslag--Solitar relation, or a relation of Adjan type; see \cite{Ivanov:2001kl, Margolis:2005il} and \cite[Section~7]{Meakin:2007zt}.  Here $w \in (A \cup A^{-1})^*$ is called an \emph{idempotent word} if it freely reduces to the identity element in the free group $FG(A)$. It is important to note that for inverse monoid presentations one cannot assume that the word $w$ in the defining relation $w=1$ is a reduced word. For example, the presentations $\Ipres{a}{}$ and $\Ipres{a}{aa^{-1}=1}$ define different monoids, the first being the free inverse monoid of rank one, and the second being the well-known \emph{bicyclic monoid} (see for instance \cite[Section~1.6]{Howie95}). 

There are several places in the literature where it is mentioned that the problem of whether inverse monoids of the form $\Ipres{A}{w=1}$ have decidable word problem remains unsolved; see e.g. \cite[Section~2.3]{Margolis:1995qo}, \cite{Hermiller:2010bs, Meakin:2007zt}.   The first place this question appears in the literature is in the paper \cite{MMS87} of Margolis, Meakin and Stephen. Indeed, in \cite[Conjecture~2]{MMS87} the following conjecture is stated: ``\emph{If $M = \Ipres{A}{w=1}$, then the word problem for $M$ is decidable.}''     The first main goal of this paper is to give some new constructions and use them to prove that, in general, this conjecture does not hold. The first main result of this paper is: 

\vspace{1.5mm}

\noindent \textbf{Theorem~A.} 
\emph{There is a one-relator inverse monoid $\Ipres{A}{w=1}$ with undecidable word problem.
} 

\vspace{1.5mm}

We shall establish Theorem~A by first proving some new results concerning the submonoid membership problem in one-relator groups. There are a number of different membership problems that have been investigated in group theory.  The most natural such problems are the subgroup membership problem (also called the generlised word problem), the rational subset membership problem, and the submonoid membership problem.  The subgroup membership problem for finitely generated groups asks: Given a finite subset $X$ of a group $G$ and an element $g \in G$, does $g$ belong to the subgroup of $G$ generated by $X$? The submonoid and rational subset membership problems are defined analogously; see Section~\ref{sec_RAAG} below for formal definitions of these decision problems.  The subgroup membership problem is a natural generalisation of the word problem.  Mihailova showed that the direct product of two copies of the free group of rank two contains a finitely generated subgroup in which the membership problem is undecidable; see \cite[Chapter IV]{Lyndon:2001lh}. On the other hand, the subgroup membership problem is decidable for free groups, and for free abelian groups. In fact, for these two classes the more general rational subset membership problem is known to be decidable as a consequence of results of Benois and Grunschlag; see e.g. \cite{Lohrey15} for further background and references. The submonoid membership problem, and rational subset membership problem, for groups have been investigated in detail in a series of papers of Lohrey and Steinberg; see e.g. \cite{Lohrey08, Lohrey11}.  

\enlargethispage*{4mm} 

In this paper we shall be specifically interested in membership problems in one-relator groups.  Since Magnus's fundamental work, many interesting results about one-relator groups have been proved; see \cite[Chapter~II, Section~5]{Lyndon:2001lh}. One-relator groups are still an active topic of research, with recent results including e.g. \cite{LouderWilton, Sapir:2011fk, Wise12}.  Several important algorithmic problems remain open for one-relator groups including the conjugacy problem, isomorphism problem, and the subgroup membership problem; see \cite[Problem~18 and Problem~19]{BooneWP}.  Not much is known in general about the subgroup membership problem for one-relator groups.  Magnus's original solution to the word problem \cite{Magnus32} showed that membership is decidable in subgroups generated by subsets of the generating set.  Pride \cite{Pride76}  showed that membership can be decided in certain subgroups of two-generated one-relator groups.  The second main result of this paper concerns the more general question of whether the submonoid membership problem is decidable in one-relator groups. Specifically, we shall prove the following result. 

\vspace{1.5mm}

\noindent \textbf{Theorem~B.} 
\emph{There are one-relator groups with undecidable submonoid membership problem.}

\vspace{1.5mm}
In fact we show that there is a one-relator group with a fixed finitely generated submonoid in which membership is undecidable. 

A corollary of Theorem~B is that the rational subset membership problem is undecidable for one-relator groups in general.  Theorem~B will be used to prove Theorem~A, but we stress that Theorem~A is not an immediate corollary of Theorem~B. A new construction is needed which encodes the submonoid membership problem from Theorem~B into the word problem of a one-relator inverse monoid. 

The general fact that there is a connection between the word problem for inverse monoids, and the submonoid membership problem for groups, is something that was first observed by Ivanov, Margolis and Meakin in \cite{Ivanov:2001kl}. Their result \cite[Theorem~3.3]{Ivanov:2001kl}  implies that in the case that the monoid $\Ipres{A}{w=1}$ is $E$-unitary (this will be defined in Section~\ref{sec_inverse}) it has decidable word problem if its maximal group homomorphic image $\Gpres{A}{w=1}$ has decidable prefix membership problem. They also prove that if $w$ is a cyclically reduced word then $\Ipres{A}{w=1}$ is $E$-unitary. These results have subsequently been applied to solve the word problem for certain families of $E$-unitary one-relator inverse monoids $\Ipres{A}{w=1}$; see \cite{Ivanov:2001kl, Margolis:2005il, Meakin:2007zt, Arye2012}.

Theorem B arose from consideration of the natural question of which right-angled Artin groups arise as subgroups of one-relator groups.  Given any finite graph the associated right-angled Artin group is the group defined by a presentation with generating set the vertices of the graph, and defining relations specifying that two generators commute if they are joined by an edge in the graph.  These groups are also known as \emph{graph groups}, and \emph{partially commutative groups}. They were originally introduced by Baudisch, and since then this class has  attracted a lot of attention in geometric group theory; see \cite{Charney072, Vogtmann15}. There are known interesting connections between one-relator groups and right-angled Artin groups, for example right-angled Artin groups arise in Wise's solution to Baumslag's conjecture about residual finiteness of one-relator groups; see \cite{Wise12}. 

Clearly right-angled Artin groups give a common generalisation of free groups, where the defining graph has no edges, and free abelian groups, where the graph is complete.  Now by Magnus's \Fz \ free groups occur commonly and naturally as subgroups of one-relator groups. On the other hand, not all free abelian groups embed in one-relator groups. Moldavanski \cite{Moldavanski67} proved that a non-cyclic abelian subgroup of a one-relator group is either free abelian of rank two, or is locally cyclic. In particular, this answers the question of which finitely generated abelian groups embed in one-relator groups.  In light of these results, it is not unreasonable to ask more generally which finitely generated right-angled Artin groups arise as subgroups of one-relator groups.  In this paper we shall show that for any finite forest $F$ the right angled Artin group $A(F)$ embeds into a one-relator group. When combined with results of Lohrey and Steinberg from \cite{Lohrey08}, the existence of one-relator groups embedding these right-angled Artin groups will allow us to prove Theorem~B.

The paper is structured in the following way.  In Section~\ref{sec_RAAG} we give some basic background and definitions concerning right-angled Artin groups, we show that $A(F)$ embeds into a one-relator group for any finite forest $F$, and then we use this to prove Theorem~B (see Theorem~\ref{thm_LSApplication}). We begin Section~\ref{sec_inverse} with some preliminaries on the theory of inverse monoid presentations. Then we give a new general construction in Theorem~\ref{thm_ConstructionGeneral}, which is then combined with Theorem~B in order to prove Theorem~A (see Theorem~\ref{main_thm_WP}).  We conclude the paper in Section~\ref{sec_conclusions} with a discussion of some open problems and directions for possible future research. 

\section{One-relator groups with undecidable submonoid membership problem}
\label{sec_RAAG}

We assume that the reader has familiarity with basic notions from group theory; see e.g. \cite{Lyndon:2001lh}. 

\subsection*{Right-angled Artin subgroups of one-relator groups}

Since later we shall be working both with inverse monoid presentations and group presentations, to avoid any confusion we shall use $\Gpres{A}{R}$ to denote the group defined by the presentation with generators $A$ and defining relations $R$.  Let $\Gamma$ be a finite simplicial graph with vertex set $V\Gamma$ and edge set $E\Gamma$. So $E\Gamma$ is a set of two-element subsets of $V\Gamma$. The \emph{right-angled Artin group} $A(\Gamma)$ associated with the graph $\Gamma$ is the group defined by the presentation 
\[
\Gpres{V\Gamma}{uv = vu \ \mbox{if and only if $\{u,v\} \in E\Gamma$}}.
\]

Given a finite simplicial graph $\Gamma$, and an isomorphism $\psi: \Delta_1 \rightarrow \Delta_2$ between two finite induced subgraphs of $\Gamma$, we use $A(\Gamma, \psi)$ to denote the HNN-extension of $A(\Gamma)$ with respect to the isomorphism between the subgroups $A(\Delta_1)$ and $A(\Delta_2)$ of $A(\Gamma)$ that is induced by $\psi$. This is a well-defined HNN-extension since by standard results on right-angled Artin groups (see for example \cite{Charney072}) the subgroups $A(\Delta_1)$ and $A(\Delta_2)$ each naturally embed into $A(\Gamma)$, and thus $\psi$ induces an isomorphism between $A(\Delta_1)$ and $A(\Delta_2)$. Therefore, by \emph{the HNN-extension $A(\Gamma, \psi)$ of $A(\Gamma)$ with respect to $\psi:\Delta_1 \rightarrow \Delta_2$} we mean the group defined by the presentation 
\begin{align*}
\mathrm{Gp}\langle V\Gamma, t \; \mid \; 
& uv = vu \ \mbox{if and only if $\{u,v\} \in E\Gamma$}, \\
& txt^{-1} = \psi(x) \ \mbox{for all $x \in V\Delta_1$}  
\rangle. 
\end{align*}
By standard results on HNN-extensions, the group $A(\Gamma)$ embeds naturally into this HNN-extension $A(\Gamma, \psi)$. Let $P_n$ denote the path with $n$ vertices. The next result shows how this construction can be used to embed $A(P_4)$ into a one-relator group.

\begin{prop}\label{prop_better}
Let $P_4$ be the graph 
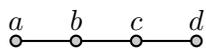
\[
\begin{tikzpicture}[thick,scale=0.8]
\tikzstyle{lightnode}=[circle, draw, fill=black!20,
                        inner sep=0pt, minimum width=4pt]
\tikzstyle{darknode}=[circle, draw, fill=black!99,
                        inner sep=0pt, minimum width=4pt]
\draw (0,0) node[lightnode] {} --(1,0);
\draw (1,0) node[lightnode] {} --(2,0);
\draw (2,0) node[lightnode] {} --(3,0) node[lightnode] {};
\draw (0,0) node[above] {$a$};
\draw (1,0) node[above] {$b$};
\draw (2,0) node[above] {$c$};
\draw (3,0) node[above] {$d$};
\end{tikzpicture}
\]
let $\Delta_1$ be the subgraph induced by $\{a,b,c\}$, let $\Delta_2$ be the subgraph induced by $\{b,c,d\}$, and let $\psi:\Delta_1 \rightarrow \Delta_2$ be the isomorphism mapping $a \mapsto b$, $b\mapsto c$, and $c \mapsto d$. Then the HNN-extension $A(P_4,\psi)$ of $A(P_4)$ with respect to $\psi$ is isomorphic to the one-relator group 
\[
\GO.
\]
Furthermore, $A(P_4)$ embeds into this one-relator group via the mapping induced by $a \mapsto a$, $b \mapsto tat^{-1}$, $c \mapsto t^2 a t^{-2}$, and $d \mapsto t^3 a t^{-3}$.  
\end{prop}
\begin{proof}
The group $A(P_4)$ is defined by the presentation 
\[
\Gpres{a,b,c,d}{
ab=ba, bc=cb, cd=dc},
\]
and $A(P_4, \psi)$ is defined by the presentation
\begin{align*}
\Gpres{a,b,c,d,t}{
ab=ba, bc=cb, cd=dc, tat^{-1}=b, tbt^{-1}=c, tct^{-1} = d
}.
\end{align*}
We now perform some Tietze transformations to show this is the one-relator group given in the statement of the proposition. Eliminating the redundant generators $d$, $c$ and $b$, in this order yields  
\begin{align*}
& \mathrm{Gp}\langle a, t \:|\: a(tat^{-1}) = (tat^{-1})a, 
(tat^{-1})(t^2 a t^{-2}) = (t^2 a t^{-2})(tat^{-1}), \\
& (t^2 a t^{-2})(t^3 a t^{-3}) = (t^3 a t^{-3})(t^2 a t^{-2}) \rangle.
\end{align*}
The last two relations are consequences of the first, obtained via conjugation by $t$, and therefore they are redundant and can be removed. This shows that $A(P_4, \psi)$ is isomorphic to 
\begin{align*}
& \GO. 
\end{align*}
Since $A(P_4)$ embeds naturally in the HNN-extension $A(P_4, \psi)$, this completes the proof.
\end{proof}

In \cite[Theorem 1.8]{Kim13} it is shown that if $F$ is any finite forest then $A(F)$ embeds into $A(P_4)$. Combined with the above proposition this gives the following result. 

\begin{thm}\label{corforest} For any finite forest $F$, the right angled Artin group $A(F)$ embeds into a one-relator group.  \end{thm}

\begin{remark}\label{rem:converse} 
While it is not important for the main results of this paper, it is worth noting that, in fact, the converse of Theorem~\ref{corforest} is also true.  That is, a right-angled Artin group $A(\Gamma)$ embeds into some one-relator group if and only if $\Gamma$ is a forest.  For one-relator groups with torsion this follows from the fact that they are hyperbolic.  The argument in the torsion-free case was pointed out by Jim Howie \cite{JimHowiePersonalCommunication} after reading an earlier version of the present article.  His argument makes use of a result of Louder and Wilton from \cite{LouderWiltonCanada2017} on Betti numbers of subgroups of torsion-free one-relator groups.  The author thanks Jim Howie for allowing his argument to be reproduced here.

In more detail, suppose that $A(\Gamma)$ embeds into a one-relator group $G$.  Seeking a contradiction suppose $\Gamma$ is not a forest and let $C_n$ with $n >2$ be the smallest cycle which embeds into $\Gamma$ as an induced subgraph.  One-relator groups with torsion are hyperbolic by Newman's Spelling Theorem.  Since no hyperbolic group contains a free abelian subgroup of rank $2$ (see e.g. \cite[Proposition~5.1]{Gersten1991}) it follows that the only right-angled Artin groups which embed into one-relator groups with torsion are free groups.  Hence the one-relator group $G$ must be torsion-free.  Now in \cite{LouderWiltonCanada2017} it is proved that for any finitely generated subgroup $H$ of a torsion-free one-relator group, we have $b_2(H) \leq b_1(H) -1$, that is, the second Betti number of $H$ is strictly less than the first.  It is well-known (see for example \cite[Subsection~3.1]{Costa2012}) that $b_1(A(\Gamma))=|V\Gamma|$ while  $b_2(A(\Gamma))=|E\Gamma|$.  In particular $b_1(A(C_n)) = b_2(A(C_n))$.  It follows that $A(C_n)$ does not embed into the torsion-free one-relator group $G$, which contradicts the fact that $A(\Gamma)$ embeds into $G$.  We conclude that if $A(\Gamma)$ embeds into a one-relator group then $\Gamma$ must be a forest.
\end{remark}

\subsection*{Undecidability of the submonoid membership problem}

Throughout we shall use $A^*$ to denote the free monoid over the alphabet $A$, and we use $A^+$ to denote the free semigroup. Let $G$ be a finitely generated group with a finite group generating set $X$. This means that $X \cup X^{-1}$ is a monoid generating set for $G$ and there is a canonical monoid homomorphism $\pi: (X \cup X^{-1})^* \rightarrow G$. The \emph{submonoid membership problem} for $G$ is the following decision problem:  

\

\noindent INPUT: A finite set of words $W = \{ w_1, \ldots w_m \} \subseteq (X \cup X^{-1})^*$ and a word $w \in (X \cup X^{-1})^*$.  

\

\noindent QUESTION: $\pi(w) \in \pi(W^*)$?  

\

This generalises the \emph{subgroup membership problem}, also called the \emph{generalised word problem}, for $G$ which takes the same input but asks whether $\pi(w) \in \pi((W \cup W^{-1})^*)$, that is, whether $w$ belongs to the subgroup generated by $W$. 
The submonoid membership problem is itself a special case of a more general problem called the \emph{rational subset membership problem} where the input is a finite automaton $\mathcal{A}$ over $X$, and the question is whether $\pi(w) \in \pi(L(\mathcal{A}))$ where $L(\mathcal{A})$ is the language recognised by $\mathcal{A}$.  
Alternatively, the class of rational subsets of a group is the smallest class that contains all finite subsets, and is closed with respect to the operations of union, product, and taking the submonoid generated by a set. 
We shall not be working with finite automata or regular languages in this paper.  
We refer the reader to \cite{Lohrey15} for more details on the rational subset membership problem for groups. 
There are non-uniform variants of these decision problems as well, where the subset of the group is fixed. Given a fixed subset $S$ of $G$ the \emph{membership problem for $S$ within $G$} is the decision problem with input a word $w \in (X \cup X^{-1})^*$ and question: $\pi(w) \in S$? 

The submonoid membership problem behaves well with respect to taking subgroups, in the following sense. 
Let $G$ be a finitely generated group and let $H$ be a finitely generated subgroup of $G$.  
If $G$ has decidable submonoid membership problem then so does $H$.
Also, for any finitely generated submonoid $N$ of $H$, if the membership problem for $N$ within $G$ is decidable then the membership problem for $N$ within $H$ is decidable.       
See \cite[Section~5]{Lohrey15} for more background on the closure properties of these decision problems. 

We may now state and prove the main result of this section. 

\begin{thm}\label{thm_LSApplication}
Let $G$ be the one-relator group 
$
\GO.
$
Then there is a fixed finitely generated submonoid $M$ of $G$ such that the membership problem for $M$ within $G$ is undecidable. 
\end{thm}
\begin{proof}
It was proved in \cite[Theorem~7]{Lohrey08} that there is a fixed finitely generated submonoid $N$ of $A(P_4)$ such that the membership problem for $N$ within $A(P_4)$ is undecidable. Let $\theta$ be an the embedding of $A(P_4)$ into $G$ given in Proposition~\ref{prop_better}, and let $M$ be the image of $N$ under this embedding. Then it follows that $M$ is a finitely generated submonoid of $G$ such that the membership problem for $M$ within $G$ is undecidable.   
\end{proof}

Since any finitely generated submonoid of a group is a rational subset we obtain the following corollary. 

\begin{cor}
There are one-relator groups with undecidable rational subset membership problem.  
\end{cor}
It would be interesting to try to classify those one-relator groups for which the rational subset membership problem is decidable. Similarly, it would be interesting to characterise the one-relator groups with decidable submonoid membership problem. As mentioned in the introduction, whether there are one-relator groups with undecidable subgroup membership problem also remains as an interesting open problem; see \cite[Problem~18]{BooneWP} . 

\section{One-relator inverse monoids with undecidable word problem}
\label{sec_inverse}

In this section we shall introduce a new construction which, when combined with the results from Section~\ref{sec_RAAG}, will be used to construct one-relator inverse monoids of the form $\Ipres{A}{w=1}$ with undecidable word problem. Before giving the construction and results, we first recall some background on free inverse monoids and inverse monoid presentations.
A more detailed account of combinatorial inverse semigroup theory may be found in \cite{Meakin:2007zt}.
For basic concepts from semigroup theory we refer the reader to \cite{Howie95}. 

\subsection*{Preliminaries on inverse monoid presentations}
An \emph{inverse monoid} is a monoid $M$ such that for every $m \in M$ there is a unique element $m^{-1} \in M$ satisfying $mm^{-1}m=m$ and $m^{-1} m m^{-1} = m^{-1}$. The element $m^{-1}$ is called the \emph{inverse} of $m$. It follows from the definition that for all $x,y \in M$ we have $x = x x^{-1} x$, $(x^{-1})^{-1} = x$, $(xy)^{-1} = y^{-1} x^{-1}$, and $xx^{-1} yy^{-1} = yy^{-1} xx^{-1}$. In fact, inverse monoids form a variety of algebras, in the sense of universal algebra, defined by these identities together with associativity. It follows from this that free inverse monoids exist. For any set $A$ the free inverse monoid $FIM(A)$ generated by $A$ may be concretely described in the following way. Let $A^{-1} = \{a^{-1} : a \in A \}$ be a set of formal inverses of the letters from $A$, where we assume the sets $A$ and $A^{-1}$ are disjoint. Given any word $x_1 x_2 \ldots x_n \in (A \cup A^{-1})^*$, with $x_i \in A \cup A^{-1}$ for $1 \leq i \leq n$, we define the formal inverse of this word to be $(x_1 x_2 \ldots x_n)^{-1} = x_n^{-1} \ldots x_2^{-1} x_1^{-1}$ where $(x^{-1})^{-1} = x$ for all $x \in A$, and we set $1^{-1} = 1$. Let $\nu$ be the congruence on $(A \cup A^{-1})^*$ generated by the set     
\[
\{
(ww^{-1}w, w), 
(ww^{-1}uu^{-1}, uu^{-1}ww^{-1}): u, w \in (A \cup A^{-1})^*
\}. 
\]
We call $\nu$ the \emph{Vagner congruence} on $(A \cup A^{-1})^*$. The \emph{free inverse monoid} $FIM(A)$ on alphabet $A$ is then isomorphic to $ (A \cup A^{-1})^* / \nu$. For the rest of this article we identify $FIM(A)$ with $ (A \cup A^{-1})^* / \nu$.  The \emph{inverse monoid defined by the presentation} $\Ipres{A}{u_i = v_i \; (i \in I)}$, where $u_i, v_i \in (A \cup A^{-1})^*$ for $i \in I$, is defined to be the quotient of the free inverse monoid $FIM(A)$ determined by these defining relations. Therefore, $\Ipres{A}{u_i = v_i \; (i \in I)}$ is isomorphic to the quotient $(A \cup A^{-1})^* / \sigma$ where $\sigma$ is the congruence on $(A \cup A^{-1})^*$ generated by $\nu \cup \{ (u_i, v_i) : i \in I \}$.

The theory of inverse monoid presentations has developed significantly over the last few decades.  It follows from results of Scheiblich and Munn that the word problem for free inverse monoids is decidable; see \cite[Chapter~5]{Howie95}.  That work shows that the elements of $FIM(X)$ may be represented by finite connected subgraphs of the Cayley graph of the free group $FG(X)$, which are now commonly known as \emph{Munn trees}. Other important work for the study of the word problem for finitely presented inverse monoids are the automata-theoretic methods introduced by Stephen in \cite{Stephen:1990ss}. We shall not need the details of Stephen's theory in this article, but they are needed in the original proofs of some of the background results which we use, most notably Proposition~\ref{IMM_FourPointTwo} below. For an excellent overview of Stephen's techniques and results see \cite[Section~2]{Ivanov:2001kl} and \cite{Meakin:2007zt}. 

Throughout this section $M$ will always denote an inverse monoid defined by a finite presentation  $\Ipres{A}{R}$ and $G$ will be used to denote the maximal group homomorphic image $\Gpres{A}{R}$ of $M$. For any word $u \in (A \cup A^{-1})^*$ we use $[u]_M$ to denote the image of $u$ in $M$, and we use $[u]_G$ to denote the image of $u$ in $G$. We shall use $\nat$ to denote the natural surjective homomorphism $\nat:M \rightarrow G$ defined by $\nat([u]_M) = [u]_G$ for all $u \in (A \cup A^{-1})^*$.  The inverse monoid $M$ is called \emph{$E$-unitary} if the natural homomorphism $\nat: M \rightarrow G$ is \emph{idempotent pure}, which means that $\nat^{-1}(1_G) = E(M)$ where $1_G$ is the identity element of the group $G$. Here $E(M)$ denotes the set of idempotents of the monoid $M$. There are other ways of characterising the property of being $E$-unitary. One important such characterisation is that it is equivalent to saying that the natural map $\nat:M \rightarrow G$ is injective on every $\gr$-class of $M$; see \cite[Lemma~1.5]{Margolis:1993tw} and \cite[Theorem~3.8]{Stephen:1990ss}. Here $\gr$ denotes Green's $\gr$-relation on the monoid $M$, where $(x,y) \in \gr$ if and only if $xM = yM$. This characterisation of $E$-unitarity will be used below in the proof of Theorem~\ref{thm_ConstructionGeneral}. 

If $U$ and $V$ are monoids we write $U \leq V$ to mean that $U$ is a submonoid of $V$. If $X$ is a subset of a monoid then we use $\Mgen{X}$ to denote the submonoid generated by $X$. Given any subset $W$ of the free monoid $(A \cup A^{-1})^*$, by the \emph{submonoid of $M$ generated by $W$} we shall mean the submonoid of $M$ generated by the subset $\{ [w]_M : w \in W\}$ of $M$. Similarly, the \emph{submonoid of $G$ generated by $W$} is the submonoid of $G$ generated by $\{[w]_G : w \in W\}$. 

Our main interest in this article will be inverse monoids defined by presentations of the form 
\[
\Ipres{A}{r_i = 1 \; (i \in I)}. 
\]
We use $FG(A)$ to denote the free group over $A$. We identify $FG(A)$ with the set of freely reduced words in $(A \cup A^{-1})^*$. For any word $w \in (A \cup A^{-1})^*$ we use $\mathrm{red}(w)$ to denote the word obtained by freely reducing $w$. It is well known, and straightforward to prove, that a word $e \in (A \cup A^{-1})^*$ represents an idempotent in the free inverse monoid $FIM(A)$ if and only if $\red(e)=1$ in $FG(A)$. We call words in $(A\cup A^{-1})^*$ with this property \emph{idempotent words}. For a word $w \in (A \cup A^{-1})^*$ we use $\pref(w)$ to denote the set of all prefixes of $w$. So
\[
\pref(w) = \{ w'  : w = w'w'' \ \mbox{with} \ w', w'' \in (A \cup A^{-1})^* \}. 
\]
An element $m$ of a monoid $M$ is \emph{right invertible} if there exists an element $n$ with $mn=1$. There is an obvious dual notion of an element being left invertible, and an element is \emph{invertible} if it is both left and right invertible. The invertible elements of $M$ are called the \emph{units} of $M$, and those that are right invertible are called the \emph{right units}. If $M$ is an inverse monoid generated by $A$ we say that \emph{the word $u \in (A \cup A^{-1})^*$ is right invertible in $M$} if $[u]_M$ is right invertible. Similarly we talk about words being left invertible, and invertible, in the monoid $M$. The following lemma is standard. For completeness we include a proof.   

\begin{lem}\label{lem_reduce}
Let $M$ be an inverse monoid, and let $a, b \in M$. If $ab$ is right invertible in $M$ then $abb^{-1} = a$. 
\end{lem}
\begin{proof}
Since $ab$ is right invertible it follows that $abb^{-1}a^{-1}=1$. Since idempotents commute in an inverse monoid, right multiplying by $a$ then gives $a = a(bb^{-1})(a^{-1}a) = a(a^{-1}a)(bb^{-1}) = abb^{-1}$.
\end{proof}

\begin{cor}\label{cor_good}
Let $M$ be an inverse monoid generated by a set $A$. If $x a a^{-1} y \in (A \cup A^{-1})^*$ is right invertible in $M$, where $a \in A \cup A^{-1}$ and $x, y \in (A \cup A^{-1})^*$, then $[x a a^{-1} y]_M = [x y]_M$. In particular, for every word $w \in (A \cup A^{-1})^*$, if $w$ is right invertible in $M$ then $[w]_M = [\mathrm{red}(w)]_M$. 
\end{cor}
\begin{proof}
Since $xaa^{-1}y$ is right invertible it follows that $xaa^{-1}$ is also right invertible, and hence $[xaa^{-1}]_M = [x]_M$ by Lemma~\ref{lem_reduce}, from which the result then follows.   
\end{proof}

\begin{lem}\label{prop_equivPres}
Let $e \in (A \cup A^{-1})^*$ be an idempotent word and let $r_1, r_2, \ldots, r_m \in (A \cup A^{-1})^*$. Then the inverse monoid presentations  
\[
\Ipres{A}{er_1=1, r_2=1, \ldots, r_m=1} 
\]
and 
\[
\Ipres{A}{e=1, r_1=1, r_2=1, \ldots, r_m=1} 
\]
are equivalent, that is, the identity map on $(A \cup A^{-1})^*$ induces an isomorphism between the inverse monoids defined by these two presentations.  
\end{lem}
\begin{proof}
Clearly all of the defining relations in the first presentation are consequences of the defining relations in the second presentation. For the converse, from Corollary~\ref{cor_good} it follows that $r_1 = er_1$ is a consequence of the defining relations in the first presentation, and from this it then follows that 
$r_1 = er_1 =1$ and hence $e=1$ is also a consequence of the defining relations.  \end{proof}

\subsection*{Construction and application to one-relator inverse monoids}
We shall use $H \ast K$ to denote the \emph{free product} of two groups $H$ and $K$, where we assume $H \cap K = \varnothing$. A \emph{reduced sequence of length $n$} is a list $g_1, g_2, \ldots, g_n$ ($n \geq 0$) such that $g_i \neq 1$ for all $1 \leq i \leq n$, each $g_i$ belongs to one of the factors $H$ or $K$, and $g_i \in H$ if and only if $g_{i+1} \in K$ for all $1 \leq i \leq n-1$. It is standard basic result that each element of $H \ast K$ can be uniquely written as $g_1 \ldots g_n$ where $g_1, \ldots, g_n$ is a reduced sequence; see \cite[Chapter~IV]{Lyndon:2001lh}. Below we shall refer to this as the \emph{normal form theorem for free products of groups}. The \emph{length of an element of} $H \ast K$ is defined to be the length of the unique reduced sequence representing that element. 

\begin{lem}\label{lem_IdealComplement}
Let $H$ be a group, let $U$ be the submonoid of the free product $H \ast FG(t)$ generated by $\{ t \} \cup H \cup tHt^{-1}$, and let $V$ be the submonoid of $H \ast FG(t)$ generated by $H \cup tHt^{-1}$. Then $V \leq U$ and $U \setminus V$ is an ideal in the monoid $U$. 
\end{lem}
\begin{proof}
It is obvious from the definitions that $V \leq U$. We are left with the task of proving that $U \setminus V$ is an ideal in $U$. Let $\theta$ be the surjective homomorphism $\theta: H \ast FG(t) \rightarrow FG(t)$ defined by $t \mapsto t$ and $h \mapsto 1$ for all $h \in H$. Set
\[
S =  \theta^{-1}(\{1, t, t^2, t^3, \ldots \}), \; \;
T = \theta^{-1}(\{ t, t^2, t^3, \ldots \}), \; \; \mbox{and} \; N = \theta^{-1}(1).
\]
Then $S$, $T$ and $N$ are submonoids of $H \ast FG(t)$ such that $S$ is the disjoint union $T \cup N$, and $T$ is an ideal of $S$. This holds since $\{ t, t^2, t^3, \ldots \}$ is an ideal of the monoid $\{1, t, t^2, t^3, \ldots \}$ and the preimage of an ideal, with respect to a surjective homomorphism, is itself an ideal. 

Since $H \cup tHt^{-1} \subseteq N$ it follows that $V \subseteq N$. As $t$ is in  $T$, and $T$ is disjoint from $N$, it follows that $t \in U \setminus V$. Thus $U \setminus V$ is non-empty. Let $z \in U \setminus V$. Write $z = x_1 x_2 \ldots x_m$ where $x_i \in \{t \} \cup H \cup tHt^{-1}$ for all $1 \leq i \leq m$. Since $z \not\in V$ it follows that $x_j=t$ for some $j$, but then since $T$ is an ideal of $S$, and $t \in T$, it follows that $z \in T$. This proves that $U \setminus V \subseteq T$. Note also that $U \leq S = T \cup N$ with $T \cap N = \varnothing$. Hence $V = U \cap N$ and $U \setminus V = U \cap T$. Since $T$ is an ideal of $S$ it then follows that $U \setminus V$ is an ideal of $U$. 
\end{proof}

The following lemma is a straightforward exercise. 

\begin{lem}\label{lem_basic}
Let $T$ be a submonoid of a monoid $S$ such that $S \setminus T$ is an ideal of $S$. Then for any subset  $X \subseteq S$ we have $\Mgen{X} \cap T = \Mgen{X \cap T}$. 
\end{lem}

The next result will be needed for the proof of Theorem~\ref{thm_ConstructionGeneral} below. 

\begin{lem}\label{lem_key_claim} 
Let $H$ be a group and let $W$ be a finite subset of $H$.  Let $T$ be the submonoid of $H$ generated by $W$, and let $S$ be the submonoid of $H \ast FG(t)$ generated by $\{t \} \cup H \cup tWt^{-1}$.  Then for all $h \in H$ \[ tht^{-1} \in S \Leftrightarrow h \in T.  \] 
\end{lem} 
\begin{proof} Let $h \in H$. If $h=1$ then the result clearly holds, so suppose otherwise. If $h \in T = \Mgen{W}$ then \[ tht^{-1} \in t\Mgen{W}t^{-1} = \Mgen{tWt^{-1 }} \subseteq S.  \]
For the converse, suppose that $h \in H$ and 
\[
tht^{-1} \in S = \Mgen{\{t \} \cup H \cup tWt^{-1}}. 
 \]
By Lemma~\ref{lem_IdealComplement} the complement of $\Mgen{H \cup tHt^{-1}} $ in $\Mgen{\{t \} \cup H \cup tHt^{-1}} $ is an ideal. It then follows from Lemma~\ref{lem_basic} that 
\begin{align*}
tht^{-1} &\in \Mgen{\{t \} \cup H \cup tWt^{-1}} \cap \Mgen{H \cup tHt^{-1}} \\
&= 
\Mgen{
(\{t \} \cup H \cup tWt^{-1}) \cap
\Mgen{H \cup tHt^{-1}}
} \\
&=\Mgen{H \cup tWt^{-1}} = \Mgen{H \cup t T t^{-1}}. 
\end{align*}
Since $tht^{-1} \in \Mgen{H \cup tTt^{-1}}$, and by assumption $tht^{-1} \neq 1$, we can write 
\[
tht^{-1} = x_1 x_2 \ldots x_m 
\]
where $m \geq 1$, $x_i \in H \cup tTt^{-1}$ and $x_i \neq 1$ for all  $1 \leq i \leq m$. Since each of $H$ and $tTt^{-1}$ is a submonoid of $H \ast FG(t)$ we can combine adjacent terms in this product if they both belong either to $H$ or both to $tTt^{-1}$, and we can remove any resulting terms that are equal to $1$. Repeating this if necessary, we can write  
\[
tht^{-1} = s_1 s_2 \ldots s_n \tag{$\diamond$}
\]
where $n \geq 1$,  each $s_i$ belongs to $H \cup tTt^{-1}$, where $s_i \neq 1$ for all $1 \le i \leq n$ and  $s_i \in H \Leftrightarrow s_{i+1} \in tTt^{-1}$ for all $1 \leq i \leq n-1$.   By the normal form theorem for free products of groups applied to the free product $H \ast FG(t)$, at least one of the terms $s_i$ must belong to $tTt^{-1}$. If $n>1$ then either $s_{i+1} \in H$ or $s_{i-1} \in H$ but in both cases this would imply that $s_1 s_2 \ldots s_n$ has  length at least four in the free product $H \ast FG(t)$. But this contradicts ($\diamond$) since $tht^{-1}$ has  length three in $H \ast FG(t)$. We conclude that $n=1$ and $s_1 \in tTt^{-1}$.  (An alternative way to see this is to observe that the submonoid $\Mgen{H \cup tTt^{-1}}$ of the group $H \ast FG(t)$ is in a natural way isomorphic to the free product $H \ast T$ of the monoids $H$ and $T$.)  This implies $tht^{-1} \in tTt^{-1}$ and thus $h \in T$, completing the proof of the lemma.  
\end{proof}

The following result is an important basic consequence of Stephen's procedure (see \cite{Stephen:1990ss}).

\begin{prop}(\cite[Theorem~3.2]{Stephen93}; cf. \cite[Proposition~4.2]{Ivanov:2001kl})
\label{IMM_FourPointTwo}
Let $M$ be the inverse monoid defined by
\[
\Ipres{A}{r_1=1, \ldots, r_k=1}
\]
where $r_i \in (A \cup A^{-1})^*$ for $1 \leq i \leq k$.  Then 
\[
X = 
\{ [p]_M : 
p \in 
\bigcup_{1 \leq i \leq k} \mathrm{pref} (r_i) 
\} 
\]
is a finite generating set for the submonoid of right units of $M$. 
\end{prop}
\begin{proof}
This result follows from the argument given in the second paragraph of the proof of \cite[Proposition~4.2]{Ivanov:2001kl}. We note that the statement \cite[Proposition~4.2]{Ivanov:2001kl} actually carries the additional assumption that the defining relators are cyclically reduced words. However, this is not used in their proof, and the result holds with that assumption removed. Alternatively, this proposition can be seen to be a corollary of \cite[Theorem~3.2]{Stephen93}. 
\end{proof}

Given a finite list of words $u_1, \ldots, u_m \in (A \cup A^{-1})^*$ we define  
\[
e(u_1, u_2, \ldots, u_m) = u_1 u_1^{-1} u_2 u_2^{-1} \ldots u_m u_m^{-1}.
\]
This is clearly an idempotent word. The following result gives a general construction which will later be used to construct one-relator inverse monoids with undecidable word problem. 

\begin{thm}\label{thm_ConstructionGeneral}
Let $A = \{a_1, \ldots, a_n\}$ and let $r_1, \ldots, r_m, w_1, \ldots w_k \in (A \cup A^{-1})^*$. Let $G$ be the group $\Gpres{A}{r_1=1, \ldots, r_m=1}$  and let $M$ be the inverse monoid 
\[
\Ipres{A,t}{er_1=1, r_2=1, \ldots, r_m=1} \tag{$\dagger$}
\]
where $e$ is the idempotent word 
\[
e(a_1, \ldots, a_n, tw_1t^{-1}, \ldots, tw_kt^{-1}, a_1^{-1}, \ldots, a_n^{-1}). 
\]
Let $T$ be the submonoid of $G$ generated by $W = \{w_1, \ldots, w_k\}$. Then $M$ is an $E$-unitary inverse monoid. Furthermore,  if $M$ has decidable word problem then the membership problem for $T$ within $G$ is decidable. 
\end{thm} 
\begin{proof}
We first prove that $M$ is $E$-unitary. Let $I = \{1, \ldots, m\}$ and $J = \{1, \ldots, k\}$. By Lemma~\ref{prop_equivPres} the presentation ($\dagger$) is equivalent to 
\begin{align*} 
\mathrm{Inv}\langle A, t \; \mid \; & 
r_i=1 \; (i \in I), 
\; a a^{-1} = 1, 
\; a^{-1} a = 1 \; (a \in A), \\ 
& tw_jt^{-1}tw_j^{-1}t^{-1}=1 \; (j \in J)  
\rangle.
\end{align*}
In particular we see from this that all of the elements $[a]_M$ $(a \in A)$ are invertible in $M$, and all of the elements $[t w_j t^{-1}]_M$ $(j \in J)$ are right invertible in $M$. The inverse monoid 
\begin{align}
& \Ipres{A}{ 
r_i=1 \; (i \in I), 
\; a a^{-1} = 1, 
\; a^{-1} a = 1 \; (a \in A)
} \label{eq_ipres}
\end{align}
is a group and thus is $E$-unitary.  It is well-known that free inverse monoids are $E$-unitary (see for example \cite[Theorem~1.1]{Margolis:1993tw}) so in particular the monoid $\Ipres{t}{}$ is $E$-unitary. 

In \cite[Proposition~7.1]{Jones84} Jones proves that the free product of two $E$-unitary inverse semigroups is again an $E$-unitary inverse semigroup. Jones's result can be applied to prove that the free product of two $E$-unitary inverse monoids is again an $E$-unitary inverse monoid. Alternatively, this fact can be deduced as a corollary of a  result of Stephen \cite[Theorem 6.5]{Stephen98} which gives sufficient conditions for the amalgamated free product of two $E$-unitary inverse semigroups to again be $E$-unitary. To apply Stephen's result one observes that the free product of two inverse monoids is isomorphic to the amalgamated semigroup free product where the identity elements of the two monoids are identified. It may then be verified that the hypotheses of \cite[Theorem 6.5]{Stephen98} are satisfied in this situation.  Therefore, the free product $S$ of the inverse monoid \eqref{eq_ipres} with $\Ipres{t}{}$, which has inverse monoid presentation 
\begin{align}
&\Ipres{A,t}{ 
r_i=1 \; (i \in I), 
\; a a^{-1} = 1, 
\; a^{-1} a = 1 \; (a \in A)
},
\end{align}
is also $E$-unitary. Now $K = G \ast FG(t)$ is the maximal group image of both $S$ and also of $M$. It follows that the identity mapping on $A \cup \{ t \}$ induces surjective homomorphisms $\phi, \psi$ and $\theta$ that make the following diagram commute:
\[
  \begin{tikzcd}
    S \arrow{r}{\psi} \arrow[swap]{dr}{\phi} & M \arrow{d}{\theta} \\
     & K
  \end{tikzcd}
\]
Since $S$ is $E$-unitary it follows that $M$ is $E$-unitary. Indeed, if $m \in M$ with $\theta(m) = 1$ then writing $m = \psi(s)$ for $s \in S$ we have $\phi(s) =  \theta(\psi(s))  = 1$ which since $S$ is $E$-unitary implies $s \in E(S)$. But then $m =  \psi(s) \in E(M)$ since homomorphisms map idempotents to idempotents. This completes the proof that $M$ is $E$-unitary.   

For the second part of the theorem, we shall prove the following key claim: 

\smallskip

\noindent For all $u \in (A \cup A^{-1})^*$ 
\[
[u]_G \in T \leq G 
\Longleftrightarrow 
\mbox{
$[tut^{-1}]_M$ is right invertible in $M$.
}
\tag{$\ddagger$}
\]
To prove ($\ddagger$), first suppose that $[u]_G \in T$. This means we can write 
\[
[u]_G = 
[w_{j_1} \ldots w_{j_l}]_G 
\]
where $j_q \in J$ for $1 \leq q \leq l$. Since this equation is written over the alphabet $A \cup A^{-1}$ and all these letters represent invertible elements in $M$, and all the defining relations in the presentation of $G$ also hold in $M$, it follows that 
\[
[u]_M = 
[w_{j_1} \ldots w_{j_l}]_M. 
\]
Since $[tw_jt^{-1}]_M$ is right invertible in $M$ for all $j \in J$, by applying Corollary~\ref{cor_good} it then follows that 
\begin{align*}
[tut^{-1}]_M &= [t w_{j_1} \ldots w_{j_l} t^{-1}]_M 
= [(t w_{j_1} t^{-1}) (t w_{j_2} t^{-1}) \ldots (t w_{j_l} t^{-1})]_M
\end{align*} 
which is right invertible in $M$ since it is a product of right invertible elements. 

For the proof of the converse implication of ($\ddagger$), let $U_R$ be the submonoid of right units of $M$. Let 
\[
P = \left(\bigcup_{i \neq 1} \mathrm{pref} (r_i) \right) \cup \mathrm{pref}(e r_1), 
\]
and let $X = \{ [p]_M : p \in P \}$. It follows from Proposition~\ref{IMM_FourPointTwo} that $X$ is a finite generating set for $U_R$. Since $M$ is $E$-unitary, it follows from \cite[Lemma~1.5]{Margolis:1993tw} (see also \cite[Theorem~3.8]{Stephen:1990ss}) that the canonical homomorphism $\nat: M \rightarrow K$ from $M$ onto its maximal group image $K = G \ast FG(t)$ induces an embedding of each $\mathscr{R}$-class of $M$ into $K$. In particular $\nat$ induces an embedding of the right units $U_R$, which is the $\mathscr{R}$-class of the identity element $1_M$ of $M$, into the group $K$. It follows that $U_R$ is isomorphic to the submonoid of the group $K$ generated by $\{[p]_K : p \in P\}$. From the definition of $e$ this is equal to the submonoid of $K$ generated by the set 
\[
\{ [x]_K : x \in A \cup A^{-1} \} \cup 
\{ [t]_K \} \cup
\{ [tw_j t^{-1}]_K : j \in J \}. 
\]
Since the submonoid of $K$ generated by $A \cup A^{-1}$ is $G$, it follows that $U_R$ is isomorphic to the submonoid of $K$ generated by $
\{t \} \cup G \cup  tWt^{-1}. 
$ Now suppose that $u \in (A \cup A^{-1})^*$ is a word such that $[tut^{-1}]_M$ is right invertible in $M$.  Then from above it follows that $[tut^{-1}]_G$ belongs to the submonoid of $K=G \ast FG(t)$ generated by $
\{t \} \cup G \cup  tWt^{-1}. 
$ It then follows from  Lemma~\ref{lem_key_claim} that $[u]_G \in T$. This completes the proof of the claim ($\ddagger$).  

To complete the proof of the theorem, suppose that $M$ has decidable word problem. Then there is an algorithm which for any word $w \in (A \cup A^{-1})^*$ decides whether or not $[w]_M$ is right invertible in $M$. Indeed, $[w]_M$ is right invertible if and only if $[ww^{-1}]_M = 1_M$. Therefore, given any word $u \in (A \cup A^{-1})^*$ we can decide whether or not $[tut^{-1}]_M$ is right invertible in $M$ which, by claim ($\ddagger$), is equivalent to deciding whether or not $[u]_G \in T \leq G$. Hence the membership problem of $T$ within $G$ is decidable. 
\end{proof}

Combining this construction with the results from the previous section then gives the main result of this section.   

\begin{thm}\label{main_thm_WP}
There is a one-relator inverse monoid $\Ipres{A}{w=1}$ with undecidable word problem. 
\end{thm}
\begin{proof}
Let $A = \{a,z\}$ and let $G$ be the one-relator group 
\[
\Gpres{a,z}{
azaz^{-1}a^{-1}za^{-1}z^{-1} = 1}.
\]
Let $W = \{w_1, \ldots, w_k\}$ be a finite subset of $(A \cup A^{-1})^*$ such that the membership problem for $T=\Mgen{W}$ within $G$ is undecidable. Such a set $W$ exists by Theorem~\ref{thm_LSApplication}.  Set $e =e(a, z, tw_1t^{-1}, \ldots, tw_kt^{-1}, a^{-1}, z^{-1})$. Then by Theorem~\ref{thm_ConstructionGeneral} the one-relator inverse monoid
\[
\Ipres{a,z,t}{
eazaz^{-1}a^{-1}za^{-1}z^{-1} = 1}
\]
has undecidable word problem. This completes the proof. 
\end{proof} 

Since the inverse monoid constructed in the proof of Theorem~\ref{main_thm_WP} is $E$-unitary we immediately obtain the following result. 

\begin{cor}
There is an $E$-unitary one-relator inverse monoid $\Ipres{A}{w=1}$ with undecidable word problem. 
\end{cor}

\section{Conclusions and open problems}
\label{sec_conclusions}

We have seen in Theorem~\ref{main_thm_WP} that there are one relator inverse monoids of the form $\Ipres{A}{w=1}$ with undecidable word problem. The key question for future research in this area is therefore: 
\begin{question}
For which words $w \in (A \cup A^{-1})^*$ does $\Ipres{A}{w=1}$ have decidable word problem?
\end{question} 
As mentioned in the introduction, there are many examples of words, or classes of words, for which it has been shown that $\Ipres{A}{w=1}$ has decidable word problem. For example it is true when $w$ is an idempotent word; see \cite{Margolis:1993tw}. In light of the results of Ivanov, Margolis and Meakin \cite{Ivanov:2001kl}, of particular importance is the case where $w$ is a reduced word. Indeed, the result  \cite[Theorem~2.2]{Ivanov:2001kl} states that if the word problem is decidable for all one relator inverse monoids $\Ipres{A}{w=1}$ with $w$ a reduced word, then it is decidable for all one-relator monoids $\Mpres{A}{u=v}$.  Note that in the examples of one-relator inverse monoids $\Ipres{A}{w=1}$ with undecidable word problem constructed in this paper, the word $w$ is not a reduced word.  Under the stronger assumption that $w$ is cyclically reduced, to show that the one-relator inverse monoid has decidable word problem it suffices, by \cite[Theorem~3.1]{Ivanov:2001kl},  to show that the corresponding one-relator group has a decidable prefix membership problem. As mentioned in the introduction, the prefix membership problem for one-relator groups has been shown to be decidable in a number of special cases. See also \cite{DolinkaGrayInPrep} for some more recent results showing that the prefix membership problem is decidable for certain classes of one-relator groups which are low down in the Magnus--Moldovanskii hierarchy.

Let $M$ be the one-relator $E$-unitary inverse monoid $\Ipres{A}{w=1}$ with undecidable word problem given by the proof of Theorem~\ref{main_thm_WP}. By Proposition~\ref{IMM_FourPointTwo} the group of units $U$ of $M$ is finitely generated, and since $M$ is $E$-unitary it follows that $U$ is isomorphic to a finitely generated subgroup of the one-relator group $\Gpres{A}{w=1}$. Therefore the group of units $U$ of $M$ has decidable word problem, while the monoid $M$ itself does not have decidable word problem. This shows that in general the word problem for one-relator inverse monoids of the form $\Ipres{A}{w=1}$ does not reduce to the word problem for their groups of units. This contrasts sharply with the situation for one-relator monoid presentations of the form $\Mpres{A}{w=1}$. Adjan proved that one-relator monoids of the form $\Mpres{A}{w=1}$ have decidable word problem. He did this by first proving that the group of unit of a special one-relator monoid $\Mpres{A}{w=1}$ is a one-relator group, and then showing that the word problem for the monoid can be reduced to solving the word problem in this group, which in turn is decidable by Magnus's theorem.  This reduction result was later extended by Makanin \cite{Makanin66} who showed that the monoid $M$ defined by the presentation $\Mpres{A}{w_1=1, \ldots, w_k=1}$ has a finitely presented group of units, and that $M$ has decidable word problem if and only if its group of units also does. Thus, the example constructed in Theorem~\ref{main_thm_WP} shows that there is no hope in general of using the same reduction to the group of units approach for the word problem for inverse monoids of the form $\Ipres{A}{w_1=1, \ldots, w_k=1}$. Further results showing the contrasting behaviour of inverse monoids defined by presentations of this form, when compared to monoids defined by such presentations, will be explored in the paper \cite{GrayRuskucInPrep} where the groups of units of these inverse monoids are investigated.

\section*{Acknowledgements}

The author thanks Jim Howie for his helpful comments in relation to Remark~\ref{rem:converse}. 
This research was supported by the EPSRC grant EP/N033353/1 ``Special inverse monoids: subgroups, structure, geometry, rewriting systems and the word problem''.

\end{document}